\documentclass{article}
\usepackage[english]{babel}
\usepackage[utf8]{inputenc}
\usepackage[T1]{fontenc}
\usepackage{geometry}
\usepackage{sectsty}
\usepackage{amsthm,amsfonts}
\usepackage{amsmath}
\usepackage{mathabx}
\sectionfont{\large}
\newtheorem{teo}{Theorem}
\newtheorem{lema}[teo]{Lemma}

\newtheorem{prop}[teo]{Proposition}
\newtheorem{defi}[teo]{Definition}
\newtheoremstyle{mytheoremstyle} % name
    {\topsep}                    % Space above
    {\topsep}                    % Space below
    {}                   % Body font
    {}                           % Indent amount
    {\scshape}                   % Theorem head font
    {.}                          % Punctuation after theorem head
    {.5em}                       % Space after theorem head
    {}  % Theorem head spec (can be left empty, meaning ‘normal’)

\theoremstyle{mytheoremstyle} \newtheorem{nota}{Remark}
\numberwithin{equation}{section}
\newcommand{\real}{\mathbb{R}}
\newcommand{\nat}{\mathbb{N}}
\newcommand \ben {\begin{equation}}
\newcommand \een {\end{equation}}
\newcommand \be {\begin{equation*}}
\newcommand \ee {\end{equation*}}
\newcommand \bi {\begin{itemize}}
\newcommand \ei {\end{itemize}}

\date{}
\title{\textbf{Stability of ground-states for a system of \textit{M} coupled semilinear Schrödinger equations}}
\author{Simão Correia\\ \textit{CMAF-UL and FCUL, Av.\ Prof.\ Gama Pinto 2,}\\
\textit{1649-003 Lisboa, Portugal}\\
\textit{Email adress: sfcorreia@fc.ul.pt}}
\begin{document}
\maketitle

\begin{abstract}
We focus on the study of the stability properties of ground-states for the system of $M$ coupled semilinear Schrödinger equations with power-type nonlinearities and couplings. Our results are generalizations of the theory for the single equation and the technique used is a simplification of the original one. Depending on the power of the nonlinearity, we may observe stability, instability and weak instability. We also obtain results for three distinct classes of bound-states, which is a special feature of the $M\ge2$ case. 
\end{abstract}

\begin{section}{Introduction}
In this work, we consider the system of $M$ coupled semilinear Schrödinger equations
\ben\tag{M-NLS}
i(v_i)_t + \Delta v_i + \sum_{j=1}^M k_{ij}|v_j|^{p+1}|v_i|^{p-1}v_i=0,\quad i=1,...,M
\een
where $V=(v_1,...,v_M):\real^+\times\real^N\to\real^M$, $k_{ij}\in\real$, $k_{ij}=k_{ji}$, and $0<p<4/(N-2)^+$ (we use the convention $4/(N-2)^+=+\infty$, if $N=1,2$, and $4/(N-2)^+=4/(N-2)$, if $N\ge 3$). Given $1\le i\neq j\le M$, if $k_{ij}\ge 0$, one says that the coupling between the components $v_i$ and $v_j$ is attractive; if $k_{ij}< 0$, it is repulsive. The Cauchy problem for $V_0\in (H^1(\real^N))^M$ is locally well-posed and, letting $T_{max}(V_0)$ be the maximal time of existence of the solution with initial data $V_0$: if $T_{max}(V_0)<\infty$, then $\lim_{t\to T_{max}(V_0)}\|\nabla V(t)\|_2=+\infty$.

In the case $M=1$ and $k=1$, we obtain the nonlinear Schrödinger equation
\ben\tag{NLS}
iv_t + \Delta v + |v|^{2p}v=0.
\een

When we look for nontrivial periodic solutions of the form $V=e^{it}U$, with $U=(u_1,...,u_M)\in (H^1(\real^N))^M$ (called bound-states), we are led to the study of the system
\begin{equation}\label{BS}
 \Delta u_i - u_i + \sum_{j=1}^M k_{ij}|u_j|^{p+1}|u_i|^{p-1}u_i=0 \quad i=1,...,M.
\end{equation}

Especially relevant, for both physical and mathematical reasons, are the bound-states which have minimal action among all bound-states, the so-called ground-states. In the scalar case, one may prove that there is a unique ground-state (modulo translations and rotations).

In a recent paper (\cite{simao}), we proved the existence of ground-states of (M-NLS) under the assumption
\ben\tag{P1}
\{U\in (H^1(\real^N))^M: \sum_{i,j=1}^M k_{ij}\|u_iu_j\|_{p+1}^{p+1}>0\}\neq\emptyset.
\een
Note that this assumption is a necessary condition for existence of bound-states, since, multiplying \eqref{BS} by $U$ and integrating over $\real^N$, one obtains
\ben
\sum_{i,j=1}^M k_{ij}\|u_iu_j\|_{p+1}^{p+1}=\sum_{i=1}^M \|\nabla u_i\|_2^2 + \|u_i\|_2^2 >0.
\een
Therefore (P1) is equivalent to the existence of ground-states. To prove this, we did not use Schwarz symmetrization, since such an approach would only work if the coupling coefficients were positive. A careful application of the concentration-compactness principle turns out to be the right answer. Moreover, under fairly large conditions, we characterized the set of ground-states. More precisely, if one may group the components in such a way that two components attract each other if and only if they are in the same group, then only one of these groups is nontrivial, and it must have the same profile as the ground-state for the scalar equation.

Regarding stability, the scalar case was been treated in \cite{cazenave}, \cite{straussshatah}, \cite{grillakis}, among others. For ground-states, stability is equivalent to the condition $p<2/N$ (called the subcritical case). Note that, from the gauge and translation invariances, one should study the orbital stability of ground-states (that is, modulo rotations and translations). In \cite{cazenave}, it is possible to find examples which show that one must really consider this kind of stability.

For the general case of bound-states, the problem is much more difficult. It can be seen that, assuming non-degeneracy, the orbital stability of a bound-states is directly related with the Morse index of the action at the bound-state (see \cite{grillakis}). If this index is 1, then stability is again equivalent to the condition $p<2/N$. If the index is greater than 1, the problem remains open. Note that the assumption of non-degeneracy is not always true, even for ground-states: in \cite{simao}, we proved that, for $M=2$, $k_{11}=k_{12}=k_{22}=1$ and $p=1$, there exist a continuum of ground-states which are not related by gauge invariance. This situation, though somewhat excepcional, shows that one cannot use \textit{a priori} the results of \cite{grillakis}.

In this work, we show the analogous stability results for ground-states of (M-NLS), assuming only (P1). This was done for $M=2$ and $k_{ij}>0$ in \cite{maiamontefuscopellacci}. The framework will be very close to the scalar case as is \cite{cazenave}, though some subtle changes will be done. Specifically, to prove stability (or instability), one proves that the set of ground-states is the set of minimizers of an adequate minimization problem. This is done in two steps:
\begin{enumerate}
\item Prove that the minimization problem has a solution, independently of the existence of ground-states;
\item Using the solution found in the previous step, show the equivalence between ground-states and minimizers.
\end{enumerate}
Here, we change the argument. We shall prove directly that ground-states are minimizers and conclude the equivalence. This is more efficient, since the proof of existence of minimizers without using the ground-states and assuming only (P1) has to go through the concentration-compactness principle, which is not trivial at all (see \cite{simao}). Furthermore, we define three different classes of bound-states and prove stability results for these solutions. These are generalizations of the results obtained in \cite{maiamontefuscopellacci}.
\end{section}
\begin{section}{Definitions and main results}

Given any $U=(u_1,...,u_M)\in (H^1(\real^N))^M$, define the following functionals:
\ben
M(U):=\sum_{i=1}^M \|u_i\|_2^2,\quad T(U):=\sum_{i=1}^M \|\nabla u_i\|_2^2,\quad J(U):=\sum_{i,j=1}^M k_{ij}\|u_iu_j\|_{p+1}^{p+1},
\een
\ben
I(U):= M(U) + T(U), \quad E(U):=\frac{1}{2}T(U)-\frac{1}{2p+2}J(U),\quad H(U):=T(U) - \frac{Np}{2p+2}J(U).
\een
Finally, define the action of $U$
\ben
S(U)=\frac{1}{2}I(U)-\frac{1}{2p+2}J(U).
\een
\begin{nota}
The functional $M$ is called \textit{mass}, $T$ is the \textit{kinetic energy} and $J$ is the \textit{potential energy}. Obviously, $E$ is the total energy (or just energy). Notice that (M-NLS) may be written in a Hamiltonian way:
\ben\label{eqhamiltoniana}
iU_t=E'(U).
\een
From this, one easily observes the conservation of the $L^2$-norm of each component (and therefore of the mass) and of the energy for (M-NLS): multiply \eqref{eqhamiltoniana} by $iU$ and $U_t$, respectively, take the real part and sum in $i$ for the latter case.
\end{nota}
\begin{nota}\label{P}
Consider, for $U\in (H^1(\real^N))^M$ and $\lambda>0$, $\mathcal{P}(U,\lambda)(x)=\lambda^{\frac{N}{2}}U(\lambda x)$. By a change of variables, one sees that
\ben
M(\mathcal{P}(U,\lambda))=M(U).
\een
Now, differentiating $S(\mathcal{P}(U,\lambda))$ with respect to $\lambda$, 
\ben\label{derivadaaccao}
\frac{d}{d\lambda}S(\mathcal{P}(U,\lambda))=H(\mathcal{P}(U,\lambda)).
\een
\end{nota}
\begin{nota}
As in the scalar case, one may prove the Virial identity for (M-NLS): given $V=(v_1,...,v_M): [0,T)\to (H^1(\real^N))^M$ solution of (M-NLS), one has
\ben\label{virial}
\frac{d^2}{dt^2}\sum_{i=1}^M \|xv_i(t)\|_2^2 = 8H(V(t)).
\een
The quantity $\sum_{i=1}^M \|xv_i(t)\|_2^2$ is called the variance of $V(t)$. This identity will be essential when proving instability.
\end{nota}
\begin{defi}
We say that $U\in (H^1(\real^N))^M$ is a bound-state of (M-NLS) if it is a nonzero solution of \eqref{BS}. Furthermore, $U$ is a ground-state if $S(U)\le S(W)$, for any bound-state $W$. The set of bound-states (resp. ground-states) will be noted by $A$ (resp. G).
\end{defi}

\begin{nota}
If $U\in A$, then, multiplying \eqref{BS} by $U$ and integrating over $\real^N$, $I(U)=J(U)$. Moreover, from Pohozaev's identity,
\ben
H(U)=0.
\een
This may also be readily seen from the Virial identity.
\end{nota}

\begin{defi}
We note by $R$ the set of bound-states such that all nonzero components are equal to the same  ground-state of (NLS), up to scalar multiplication and rotation. 
\end{defi}

\begin{defi}
Fix $X\subset\{1,...,M\}$. An element $U\in A$ belongs to $G_X$ if the vector of its nonzero components is a ground-state for the (L-NLS) system for by the $i$-th components, with $i\in X$ and $L=|X|$.
\end{defi}

\begin{nota}
It is known (see \cite{cazenave}) that, up to rotations and translations, there exists a unique ground-state for (NLS), which we note by $Q$. An element in $R$ must therefore be of the form
\ben
U=(a_ie^{i\theta_i}Q(\cdot+y)),
\een
for some $a_i\ge 0$, $\theta_i\in \real$ and $y\in\real^N$.
\end{nota}

Now we present some results of \cite{simao} that wil be used later.

\begin{lema}\label{existencia}
Assume (P1). Define
\ben
\lambda_G:=\left(\inf_{J(U)=1} I(U)\right)^{\frac{p+1}{p}}>0.
\een
Then the minimization problem
\ben\label{minlambda}
I(U)=\min_{J(W)=\lambda_G} I(W),\quad J(U)=\lambda_G
\een
has a solution and $G$ is the set of its solutions. Moreover, any minimizing sequence strongly converges to an element in $G$.
\end{lema}

\begin{lema}\label{caracterizacao}
Suppose (P1) and that there exists a partition $\{Y_k\}_{1\le k\le K}$ of $\{1,...,M\}$ such that, given $1\le i\neq j\le M$,
\ben
k_{ij} \ge 0 \mbox{ if and only if } \exists k: i,j\in Y_k.
\een
Then, if $U^0=(u_1^0,...,u^0_M)\in G$, there exists $k\in\{1,...,K\}$ such that $u^0_i=0, \forall i\notin Y_{k}$ and  $G\subset R$.
\end{lema}

\begin{nota}
Without the hypothesis in the above lemma, there may exist situations where $R$ is empty.
\end{nota}

\begin{lema}\label{gagliardo}
The optimal constant for the vector-valued Gagliardo-Nirenberg inequality
\ben\label{GN}
J(W)\le CM(W)^{p+1-\frac{Np}{2}}T(W)^{\frac{Np}{2}},\ W\in (H^1(\real^N))^M
\een
is
\ben
C_M=\frac{J(\mathcal{Q})}{M(\mathcal{Q})^{p+1-\frac{Np}{2}}T(\mathcal{Q})^{\frac{Np}{2}}},\ \mathcal{Q}\in G.
\een
Moreover, one has equality if and only if
\ben
\nu W(\zeta x) \in G,
\een
where
\ben\label{defmu}
\nu = \left(\frac{J(\mathcal{Q})M(W)}{M(\mathcal{Q})J(W)}\right)^{\frac{1}{2p}}
\een
and
\ben\label{defzeta}
\zeta = \left(\nu^2\left(\frac{M(W)}{M(\mathcal{Q})}\right)\right)^{\frac{1}{N}}.
\een

\end{lema}

When searching for ground-states for (NLS), one may adopt two strategies: the first is the one presented in lemma \ref{existencia}; the second is to minimize the energy, fixing the mass equal to some constant. Then, using a suitable scaling determined by the associated Lagrange multiplier, one obtains a ground-state. For a precise value of this constant, the multiplier is $1$ and so minimizers are ground-states. Note that this only works if $p<2/N$.

We can try to adopt a similar strategy for the (M-NLS) system, for $p<2/N$. There are two ways of extending such a procedure:
\bi
\item Minimize the energy, fixing the total mass equal to some constant. We show (lemma \ref{equivalenciasubcritico}) that this is equivalent to the minimization problem \eqref{minlambda};
\item Minimize the energy, fixing the mass of \textit{each} component equal to some positive constant. More precisely, given $c>0$, consider the minimization problem
\ben\label{minenergiamassadecadaumafixa}
E(U)=\min_{\{W:\|w_i\|_2^2=c\}} E(W),\quad \|u_i\|_2^2=c\ \forall i.
\een
However, it is not necessary that one even obtains bound-states, since there will exist $M$ Lagrange multipliers which may be different, and so it is not possible to make a scaling to obtain a ground-state (notice that the minimizers will correspond to periodic solutions of the form $U=(e^{i\omega_it}u_i)_{1\le i\le M}$. If $\omega_i\neq\omega_j$, the corresponding components will be out of phase). 
\ei

\begin{defi}
We define $B^c$ to be the set of minimizers of \eqref{minenergiamassadecadaumafixa} that belong to $A$. For $X\in\{1,...,M\}$, we define $B^c_X$ to be the set of elements for which the vector of its nonzero components is in $B^c$ for the (L-NLS) system formed by the $i$-th components, with $i\in X$ and $L=|X|$.
\end{defi}

\begin{defi}
Let $\mathcal{S}\subset (H^1(\real^N))^M$ be invariant by the flow generated by (M-NLS).
We say that $\mathcal{S}$ is:
\begin{enumerate}
\item stable if, for each $\delta>0$, there exists an $\epsilon>0$ such that, for any $V_0\in (H^1(\real^N))^M$ with
\ben
\inf_{W\in\mathcal{S}} \|V_0-W\|_{H^1(\real^N)^M}<\epsilon,
\een
the solution $V$ of (M-NLS) with initial data $V_0$ satisfies
\ben
\inf_{W\in\mathcal{S}} \|V(t)-W\|_{H^1(\real^N)^M}<\delta, \forall t<T_{max}(V_0).
\een
\item weakly unstable if there exist $\epsilon>0$ and a sequence $V_n^0$ such that
\ben
\inf_{U\in \mathcal{S}} \|V_n^0-U\|_{H^1(\real^N)^M} \to 0, \ n\to\infty
\een
and, letting $V_n$ be the solution of (M-NLS) with initial data $V_n^0$,
\ben
\sup_{t\in [0,T_{max}(V_n^0))} \inf_{U\in \mathcal{S}} \|V_n(t)-U\|_{H^1(\real^N)^M}>\epsilon.
\een
\item unstable if, for any $U\in \mathcal{S}$, there exists a sequence $U_n\to U$ such that $T_{max}(U_n)<\infty$, for any $n\in\nat$.
\end{enumerate}
\end{defi}

Next, we present the main results of this paper:
\begin{teo}\label{stability}
Assume (P1) and $p<2/N$. For any $X\subset \{1,...,M\}$, let $G_1\subset G_X$ be such that $\mbox{dist}(G_1, G_X\setminus G_1)>\delta$, for some $\delta>0$. Then $G_1$ is stable.
\end{teo}
\begin{nota}
In many cases, the set $G$ is discrete modulo translations and rotations. Then any connected component of $G$ is stable. Since these components are obtained by translations and rotations of a given element, one obtains orbital stability of ground-states.
\end{nota}

\begin{prop}\label{stabilityB}
Suppose (P1), $p<2/N$, $k_{ij}>0, i\neq j$ and that there exists $\beta>0$ such that
\ben
\quad \sum_{j=1}^M k_{ij} = \beta, \ \forall i.
\een
Then, for $c=\|\beta^{\frac{1}{2p}}Q\|_2^2$,
\ben
B^c=\{(e^{i\theta_i}\beta^{\frac{1}{2p}}Q(\cdot+y))_{1\le i\le M}: \theta_i\in\real, y\in\real^N\}\subset R
\een
and, given $X\subset\{1,...,M\}$, $B^c_X$ is stable. 
\end{prop}

\begin{nota}
The above results show the existence of stable bound-states that are not ground-states.
\end{nota}

\begin{teo}\label{stronginstability}
Assume (P1) and $p>2/N$. Then $G$ and $R$ are unstable.
\end{teo}
\begin{nota}
Under the assumptions of lemma \ref{caracterizacao}, it is sufficient to prove that $R$ is unstable, and this follows from the instability of the ground-states for (NLS).
\end{nota}
\begin{teo}\label{weakinstability}
Assume (P1) and $p>2/N$. If $U\in A$ is a local minimum of $S$ over the set
\ben
\mathcal{H}:=\{W:\ H(W)=0\},
\een
then the set $\{e^{i\theta}U(\cdot + y): \theta\in\real,\ y\in\real^N\}$, is weakly unstable.
\end{teo}

\begin{nota}
Set $p>2/N$. Assuming that $U\in A$ is a non-degenerate critical point of the action (modulo rotations), the Morse index of $S$ at $U$, $m(U)$, is greater or equal to $1$: a negative direction is given by the path $\lambda\mapsto S(\mathcal{P}(U,\lambda))$. Note that this direction does not belong to the tangent space of $\mathcal{H}$ at $U$. Therefore, the condition in the above theorem is equivalent $m(U)=1$. The problem for $m(U)\ge 2$ is much more difficult, and it is still unanswered for the scalar equation.
\end{nota}

\begin{teo}\label{completeinstablity}
Assume (P1) and $p=2/N$. Then $A$ is unstable.
\end{teo}

\end{section}

\begin{section}{Stability in the subcritical case}

Throughout this section, we shall assume $p<2/N$.
\begin{lema}
There exists $\mu>0$ such that
$$
M(\mathcal{Q})=\mu,\ \forall \mathcal{Q}\in G.
$$
\end{lema}
\begin{proof}
This follows easily from the identities $I(\mathcal{Q})=J(\mathcal{Q})$ and Pohozaev's identity.
\end{proof}

\begin{lema}\label{equivalenciasubcritico}
Assume (P1). Then $G$ is the set of of solutions of the minimization problem
\ben\label{minenerg}
E(U)=\min_{M(W)=\mu} E(W), \quad M(U)=\mu.
\een
Moreover, if $\{W_n\}$ is a minimizing sequence, then $J(W_n)\to J(\mathcal{Q})$, with $\mathcal{Q}\in G$.
\end{lema}
\begin{proof}
Let $W$ be such that $M(W)=\mu$. Consider the function (see remark \ref{P})
\ben
\lambda\mapsto f(\lambda)=E(\mathcal{P}(W,\lambda)),\ \lambda>0
\een 
Since $p<2/N$, $f$ has a unique minimum $\lambda_0$. Let $Z=f(\lambda_0)$. Then $f'(\lambda_0)=0$, which implies that $H(Z)=0$, i.e.,
\ben
T(Z)=\frac{Np}{2p+2}J(Z).
\een
Therefore,
\ben
E(Z)=\frac{Np-2}{2Np}T(Z).
\een
Using the vector-valued Gagliardo-Nirenberg inequality,
\ben\label{gagli}
\frac{2p+2}{Np}T(Z)=J(Z)\le C_MM(Z)^{\frac{2-(N-2)p}{2}}T(Z)^{\frac{Np}{2}},
\een
and so, from $M(Z)=\mu$,
\ben
\frac{2p+2}{Np}T(Z)^{\frac{2-Np}{2}}\le C_M\nu^{\frac{2-(N-2)p}{2}}.
\een
Let $\mathcal{Q}\in G$. By lemma \ref{gagliardo}, we obtain $T(Z)\le T(\mathcal{Q})$. Therefore
\ben
E(W)\ge E(Z) = \frac{Np-2}{2Np}T(Z) \ge \frac{Np-2}{2Np}T(\mathcal{Q}) =E(\mathcal{Q})
\een
and so $\mathcal{Q}$ is a solution of \eqref{minenerg}. If $W$ is also a solution of \eqref{minenerg}, then one must have equality in \eqref{gagli}. Again by lemma \ref{gagliardo}, 
\ben
\nu W(\zeta x) \in G,
\een
with $\nu, \zeta$ given by \eqref{defmu}, \eqref{defzeta}. Since $M(W)=M(\mathcal{Q})$ and
$J(W)=J(\mathcal{Q})$, $\nu=\zeta=1$. Therefore $W\in G$.

If $\{W_n\}_{n\in\nat}$ is a minimizing sequence, define $\{Z_n\}_{n\in\nat}$ as above. Then $\{Z_n\}$ is also a minimizing sequence and
 $$\|W_n-Z_n\|_{(H^1(\real^N))^M}\to 0,\ n\to\infty.$$Hence
\ben
\frac{Np-2}{2(2p+2)}J(\mathcal{Q})=E(\mathcal{Q})=\lim E(Z_n) = \lim \frac{Np-2}{2(2p+2)}J(Z_n)=\lim \frac{Np-2}{2(2p+2)}J(W_n),
\een
as we wanted.
\end{proof}

\textbf{\textit{Proof of theorem \ref{stability}}:}
We start with the stability for $X=\{1,...,M\}$ (that is, for $G$). 
By contradiction suppose that there exists a sequence $\{V^0_n\}_{n\in\nat}\subset (H^1(\real^N))^M$ such that, for some $\mathcal{Q}_0\in G_1$,
\ben
\|V^0_n-\mathcal{Q}_0\|_{(H^1(\real^N))^M}\to 0,\ n\to\infty
\een
and, letting $V_n$ be the solution of (M-NLS) with initial data $V^0_n$, there exist $\{t_n\}_{n\in\nat}$ and $\epsilon>0$ such that
\ben\label{afastado}
\inf_{\mathcal{Q}\in G_1} \|V_n(t_n)-\mathcal{Q}\|_{(H^1(\real^N))^M} = \epsilon.
\een
By continuity and conservation of mass and energy,
\ben
E(V_n(t_n))=E(V_n^0)\to E(\mathcal{Q}_0),\ M(V_n(t_n))=M(V_n^0)\to M(\mathcal{Q}_0)=\mu.
\een
Therefore, the sequence 
\ben
W_n=\left(\frac{\mu}{M(V_n^0)}\right)^{\frac{1}{2}}V_n(t_n)
\een
is a minimizing sequence of \eqref{minenerg}. By lemma \ref{equivalenciasubcritico}, $J(W_n)\to J(\mathcal{Q}_0)$. From lemma \ref{existencia}, $W_n\to\mathcal{Q}_1$, with $\mathcal{Q}_1\in G$, which implies that $V_n(t_n)\to \mathcal{Q}_1$. Taking $\epsilon<\delta$, one obtains $d(\mathcal{Q}_1,G_1)<\delta$, which means that $\mathcal{Q}_1\in G_1$, which is absurd.

In the general case, given $X\subset\{1,...,M\}$, one may proceed exactly as above: for $i\notin X$, since the mass of each component is conserved, the $i$-th components must converge to $0$ in $L^2$ and, by interpolation, to $0$ in $L^{2p+2}$. This means that the remaining components are a minimizing sequence of \eqref{minenerg}, for the (L-NLS) system formed by the components in $X$, and therefore must converge to a ground-state of such a system.
\vskip10pt

\textit{\textbf{Proof of proposition \ref{stabilityB}}}:

As in the previous proof, we start with $X=\{1,...,M\}$. The general case $X\subset\{1,...,M\}$ is treated as in the previous proof.

Define, for $u\in H^1(\real^N)$,
\ben
E_1(u)=\frac{1}{2}\|\nabla u\|_2^2 - \frac{\beta}{2p+2}\|u\|_{2p+2}^{2p+2}.
\een

By lemma \ref{equivalenciasubcritico} for $M=1$, the set of solutions of the minimization problem
\ben
E_1(u)=\min_{\|w\|_2^2=\|\beta^{\frac{1}{2p}}Q\|_2^2} E_1(w),\quad \|u\|_2^2=\|\beta^{\frac{1}{2p}}Q\|_2^2
\een
is $\{e^{i\theta}Q(\cdot+y): \theta\in\real, y\in\real^N\}$. Let $U=(u_1,...,u_M)\in (H^1(\real^N))^M$ be such that $ \|u_i\|_2^2=\|\beta^{\frac{1}{2p}}Q\|_2^2$. Then
\ben
\sum_{i=1}^M E_1(u_i) \ge \sum_{i=1}^M E_1(Q).
\een
Let $\mathcal{Q}$ be the vector formed by $M$ copies of $Q$. Now, from Young's inequality, we have
\ben
E(U)\ge \sum_{i=1}^M E_1(u_i)  \ge \sum_{i=1}^M E_1(Q) = E(\mathcal{Q}).
\een
Therefore $\mathcal{Q}$ is a solution of \eqref{minenergiamassadecadaumafixa}. If $U$ is also a solution, one must have equality in the above relation, which implies that
\ben
u_i=e^{i\theta_i}Q(\cdot+y_i), \ \theta_i\in\real, \ y_i\in\real^N.
\een

If there exist $i_0, j_0$ such that $y_{i_0}\neq y_{j_0}$, one easily sees that there exists $D\subset \real^N$ of positive measure such that, for all $x\in D$, $Q(x+y_{i_0})\neq Q(x+y_{j_0})$ and so, using Young's inequality,
\ben
Q(x+y_{i_0})^{p+1}Q(x+y_{j_0})^{p+1}< \frac{1}{2}Q(x+y_{i_0})^{2p+2} + \frac{1}{2}Q(x+y_{j_0})^{2p+2},\quad x\in D.
\een
On the other hand, we have in general
\ben
Q(x+y_{i})^{p+1}Q(x+y_{j})^{p+1}\le \frac{1}{2}Q(x+y_{i})^{2p+2} + \frac{1}{2}Q(x+y_{j})^{2p+2},\quad x\in \real^N,\  1\le i,j\le M.
\een
Consequently,
\begin{align*}
\int (a_{i}Q(\cdot+y_{i}))^{p+1}(a_jQ(\cdot+y_j))^{p+1} \le a_i^{p+1}a_j^{p+1}\left(\frac{1}{2}\int Q(\cdot+y_i)^{2p+2} + \frac{1}{2}\int Q(\cdot+y_j)^{2p+2}\right)\\ = a_i^{p+1}a_j^{p+1}\int Q^{2p+2} = \int (a_iQ)^{p+1}(a_jQ)^{p+1},
\end{align*}
with strict inequality if $i=i_0$ and $j=j_0$ and so
\ben
E(U)=\frac{1}{2}T(U)-\frac{1}{2p+2}J(U) > \frac{1}{2}T(\mathcal{Q})-\frac{1}{2p+2}J(\mathcal{Q}) = E(\mathcal{Q}),
\een
which is absurd. Hence 
\ben
B^c=\{(e^{i\theta_i}\beta^{\frac{1}{2p}}Q(\cdot+y))_{1\le i\le M}: \theta_i\in\real, y\in\real^N\}.
\een

Now we prove the stability property. By contradiction, suppose that there exists a sequence $\{V^0_n\}_{n\in\nat}\subset (H^1(\real^N))^M$ such that, for some $\mathcal{Q}_0\in G$,
\ben
\|V^0_n-\mathcal{Q}_0\|_{(H^1(\real^N))^M}\to 0,\ n\to\infty
\een
and, letting $V_n$ be the solution of (M-NLS) with initial data $V^0_n$, there exist $\{t_n\}_{n\in\nat}$ and $\epsilon>0$ such that
\ben\label{afastado2}
\inf_{\mathcal{Q}\in B^c} \|V_n(t_n)-\mathcal{Q}\|_{(H^1(\real^N))^M} > \epsilon.
\een
By continuity and from the conservation of the $L^2$ norm of each component and of the energy,
\ben
E(V_n(t_n))=E(V_n^0)\to E(\mathcal{Q}_0),\ \|(V_n(t_n))_i\|_2^2=\|(V_n^0)_i\|_2^2\to \|\beta^{\frac{1}{2p}}Q\|_2^2.
\een
Therefore, the sequence $W_n=(w_n^1,...,w_n^M)$ defined by
\ben
w_n^i=\left(\frac{\|\beta^{\frac{1}{2p}}Q\|_2^2}{\|(V_n^0)_i\|_2^2}\right)^{\frac{1}{2}}V_n(t_n),\ i=1,...,M
\een
is a minimizing sequence of \eqref{minenergiamassadecadaumafixa}. Now notice that this implies that
\ben
E_1(w_n^i)\to E_1(Q), \ i=1,...,M.
\een
From the stability results for (NLS) (see \cite{cazenave}, chapter 8), this implies that, for some $\theta_i\in\real$ and $y_i\in\real^N$, $w_n^i\to e^{i\theta_i}Q(\cdot+y_i)$. Applying a reasoning as before, we see that $y_i=y$, for all $i$. Therefore $W_n\to \mathcal{Q}_1$, with $\mathcal{Q}_1\in B^c$ and so $V_n\to \mathcal{Q}_1$, which is absurd, by \eqref{afastado2}.

\end{section}
\begin{section}{Instability in the supercritical case}
In this section, we study the case $p>2/N$. We define, for $W\neq 0$, $\lambda^*(W)$ to be the maximum of the function $g(\lambda)=S(\mathcal{P}(W,\lambda))$.

\begin{lema}
Assume (P1). Then $G$ is the set of of solutions of the minimization problem
\ben\label{minh}
S(U)=\min_{H(W)=0} S(W), \quad H(U)=0.
\een
\end{lema}
\begin{proof}
Let $W$ be such that $H(W)=0$ and consider the function (see remark \ref{P})
\ben
\lambda\mapsto g(\lambda)=S(\mathcal{P}(W,\lambda)),\ \lambda>0
\een 
Since $p>2/N$, this function has a unique maximum and, by \eqref{derivadaaccao}, it must be $\lambda=1$. Therefore
\ben
S(\mathcal{P}(\lambda,W))\le S(W), \forall \lambda>0.
\een
On the other hand, there exists $\lambda_0>0$ such that $J(\mathcal{P}(\lambda_0,U))=\lambda_G$. Hence, for $\mathcal{Q}\in G$,
$$
S(\mathcal{Q})\le S(\mathcal{P}(\lambda_0,W))\le S(W), \forall W:\ H(W)=0.
$$
Therefore $G$ is a subset of the set of solutions of \eqref{minh} and the latter is nonempty.

Now consider $U$ solution of \eqref{minh}. Define, for $\sigma>0$, $U_\sigma(x)=\sigma^{\frac{1}{p}}U(\sigma x)$. By a change of variables,
\ben
H(U_\sigma)=\sigma^{2-N+\frac{2}{p}}H(U)=0.
\een
Since $U$ is a minimizer, one must have
\ben
\frac{d}{d\sigma} S(U_\sigma)\Big|_{\sigma=1}=0,\mbox{ i.e. }  \langle S'(U), U \rangle_{H^{-1}\times H^1} =0
\een
On the other hand, there exists $\eta$ such that $S'(U)=\eta H'(U)$. Applying to $U$ and using $H(U)=0$,
\ben
0 = \langle S'(U), U \rangle_{H^{-1}\times H^1} =\eta\langle H'(U), U \rangle_{H^{-1}\times H^1}  = -2p\eta T(U).
\een
Therefore $\eta=0$ and so $U\in A$. Given $\mathcal{Q}\in G$, $H(\mathcal{Q})=0$, and so $S(U)\le S(\mathcal{Q})$, which means that $U\in G$.
\end{proof}
\begin{lema}
Let $\mathcal{Q}\in G$ and $W\in (H^1(\real^N))^M$ such that $H(W)<0$. Then
\ben
H(W)\le S(W) - S(\mathcal{Q}).
\een
\end{lema}
\begin{proof}
Once again, consider the function
\ben
\lambda\mapsto g(\lambda)=S(\mathcal{P}(W,\lambda)),\ \lambda>0.
\een 
This function has a maximum $\lambda_0<1$ (since $H(W)<0$) and is concave in $(\lambda_0,1)$. Therefore, by remark \ref{derivadaaccao},
\ben
S(W)  \ge  S(\mathcal{P}(W,\lambda_0)) + (1-\lambda_0)H(W) \ge S(\mathcal{P}(W,\lambda_0)) + H(W) \ge S(\mathcal{Q}) + H(W), 
\een
since $H(\mathcal{P}(W,\lambda_0))=0$ and $\mathcal{Q}$ is a solution of \eqref{minh}.
\end{proof}

\begin{nota}\label{desigualdadegeral}
More generally, given any $U\in (H^1(\real^N))^M$, one may prove as above that, if $W$ is such that $H(W)<0$ and $S(\mathcal{P}(W,\lambda^*(W)))\ge S(U)$,
\ben
H(W)\le S(W) - S(U).
\een
\end{nota}

\textbf{\textit{Proof of theorem \ref{stronginstability}}:}

Firstly, we prove that $G$ is unstable.
Consider a ground-state $\mathcal{Q}$. Then, for any $\lambda>1$, $\mathcal{Q}_\lambda:=\mathcal{P}(\mathcal{Q},\lambda)$ satisfies
\ben
H(\mathcal{Q}_\lambda)<0.
\een

Let $V_\lambda$ be the solution of (M-NLS) with initial data $\mathcal{Q}_\lambda$. For $t$ small, $H(V_\lambda(t))<0$. From the conservation of mass and energy,
\ben
S(V_\lambda(t))=S(\mathcal{Q}_\lambda).
\een

By the previous lemma, for any $t$ such that $H(V_\lambda(t))<0$, one has
\ben
H(V_\lambda(t))\le S(V_\lambda(t)) - S(\mathcal{Q}) \le S(\mathcal{Q}_\lambda) - S(\mathcal{Q}) =-\delta<0.
\een
Therefore, by continuity, one must have $H(V_\lambda(t))\le-\delta, \forall t<T_{max}(\mathcal{Q}_\lambda)$. Now, using \eqref{virial},
\ben\label{explosãovirial}
\frac{d^2}{dt^2}\sum_{i=1}^M \|x(v_\lambda)_i(t)\|_2^2 = 8H(V(t))<-8\delta.
\een
Since the variance is positive, one must have $T_{max}(\mathcal{Q}_\lambda)<\infty$ and so $G$ is unstable.
\vskip10pt
If $U=(u_1,...,u_M)\in R$, then there exist $a_i\ge 0$, $\theta_i\in\real$ and $y\in\real^N$ such that $u_i=a_ie^{i\theta_i}Q(\cdot + y)$. Since $Q(\cdot + y)$ is a ground-state for (1-NLS), there exists a sequence $\{v_n^0\}_{n\in\nat}$ such that $v_n^0\to Q(\cdot+y)$ in $H^1(\real^N)$ and $T_{max}(v_n^0)<\infty, \ \forall n$. Let $v_n$ be the solution of (1-NLS) with initial data $v_n^0$. Then one can observe that $V_n=(a_ie^{i\theta_i}v_n)_{1\le i\le M}$ is a solution of (M-NLS), with initial data $V_n^0=(a_ie^{i\theta_i}v_n^0)_{1\le i\le M}$. Since $V_n^0\to U$ in $(H^1(\real^N))^M$, one concludes that $R$ is unstable.
 $\qedsymbol$
\vskip10pt

\textbf{\textit{Proof of theorem \ref{weakinstability}}:}

Let $U\in A$ be a local minimum of $S$ restricted to $\mathcal{H}$. Let $B_\delta(U)$ be a ball with center at $U$ and radius $\delta$ fixed such that
\ben
S(U)\le S(W),\ \forall W\in B_\delta(U)\cap \mathcal{H}.
\een
For $\epsilon>0$ small, one has
\ben
\mathcal{P}(W,\lambda^*(W))\in B_\delta(U),\ \forall W\in B_\epsilon(U).
\een
From remark \ref{desigualdadegeral}, if $W\in B_\epsilon(U)$ is such that $H(W)<0$,
\ben
H(W)\le S(W) - S(U).
\een
Notice that, from the invariance of $S$ and $H$ regarding rotations and translations, the same remains valid for
\ben
W\in \Sigma:= \{e^{i\theta}Z(\cdot + y): \theta\in\real, \ y\in\real^N,\ Z\in B_\epsilon(U)\}.
\een
Consider $U_\lambda=\mathcal{P}(U,\lambda),\ \lambda>1$. Then $H(U_\lambda)<0$. Let $V_\lambda$ be the solution of (M-NLS) with initial data $U_\lambda$. If $V_\lambda(t)\in \Sigma,\ \forall t<T_{max}(U_\lambda)$, then, arguing as in the previous proof,
\ben
H(V_\lambda(t))\le S(U_\lambda)-S(U)=-\delta<0,\ \forall t<T_{max}(U_\lambda).
\een
Then \eqref{explosãovirial} is valid, which leads to $T_{max}(U_\lambda)<\infty$. Since $\Sigma$ is bounded, we arrive at a contradiction. $\qedsymbol$
\end{section}
\begin{section}{Instability in the critical case}
\textit{Proof of theorem \ref{completeinstablity}}:
First, notice that $2E(W)=H(W)$, for any $W\in (H^1(\real^N))^M$.
Let $U\in A$. Then $2E(U)=H(U)=0$. For any $\lambda>1$, $H(\lambda U)<0$. Since $2E=H$, the conservation of energy implies that, setting $V_\lambda$ to be the solution of (M-NLS) with initial data $\lambda U$, $H(V_\lambda(t))=H(\lambda U)$, $t<T_{max}(\lambda U)$. One now concludes as in the supercritical case, using the Virial identity. $\qedsymbol$
\end{section}

\end{document}